\documentclass[11pt,reqno]{amsart}
\usepackage{srcltx}
\usepackage{empheq}
\usepackage{bm}
\usepackage{amsmath,amssymb,cases,color,mathtools}
\usepackage{hyperref}
\usepackage[capitalise]{cleveref}
\usepackage{vmargin}
\usepackage{bigints}
\usepackage{relsize}
\usepackage{stmaryrd}

\usepackage{amsmath, amssymb,amscd, }
\usepackage{amsfonts}
\usepackage{mathrsfs}
\usepackage{graphicx}
\usepackage{xcolor}

 \usepackage{upref}
\hypersetup{linkcolor=blue, colorlinks=true,citecolor = red}
\newtheorem {theorem}{Theorem}[section]
\newtheorem {lemma}{Lemma}[section]
\newtheorem {definition}{Definition}[section]
 
\newtheorem {remark}{Remark}[section]

\theoremstyle{definition}

\theoremstyle{remark}
\numberwithin{equation}{section}
 
\newcommand{\mc}{\mathcal}

\newcommand{\dx}{\,{\rm d} {x}}
\newcommand{\dS}{\,{\rm d} {S}}

\makeatletter
\renewcommand*\env@matrix[1][*\c@MaxMatrixCols c]{%
	\hskip -\arraycolsep
	\let\@ifnextchar\new@ifnextchar
	\array{#1}}
\makeatother

\begin{document}

\title[Self-propelled Dirichlet]{Steady Motion of a Self-Propelled Body in a Viscous Fluid: Dirichlet Boundary Conditions with Nonzero Flux}

 \author{\v S\'arka Ne\v casov\'a $^1$\and Arnab Roy$^{2,3}$ \and Ana Leonor Silvestre$^{4}$}
  \thanks{\v S.N. was supported by Premium Academia of {\v{S}}. Ne{\v{c}}asov{\'{a}}. The
Institute of Mathematics, CAS is supported by RVO:67985840.
 A.R is supported by the Grant RYC2022-036183-I funded by MICIU/AEI/10.13039/501100011033 and by ESF+. A.R has been partially supported by the Basque Government through the BERC 2022-2025 program and by the Spanish State Research Agency through BCAM Severo Ochoa CEX2021-001142-S and through project PID2023-146764NB-I00 funded by MICIU/AEI/10.13039/501100011033 and cofunded by the European Union. A. L. S. acknowledges the financial support of Funda\c{c}\~ao para a Ci\^encia e a Tecnologia (FCT), Portuguese Agency for Scientific Research, through the project
UID/PRR/04621/2025 of CEMAT/IST-ID, DOI https://doi.org/10.54499/UID/PRR/04621/2025.}

\date{}

\maketitle

\centerline{$^1$ Institute of Mathematics, Czech Academy of Sciences,}

\centerline{\v Zitn\'a 25, 115 67 Praha 1, Czech Republic.}

\centerline{$^2$ BCAM, Basque Center for Applied Mathematics}

\centerline{Mazarredo 14, E48009 Bilbao, Bizkaia, Spain.}

\centerline{$^3$IKERBASQUE, Basque Foundation for Science, }

\centerline{Plaza Euskadi 5, 48009 Bilbao, Bizkaia, Spain.}

\centerline{$^4$ Department of Mathematics, Instituto Superior T\'{e}cnico,}

\centerline{1049-001 Lisboa, Portugal.}

\bigskip
\bigskip
 
 \begin{center} \it In memory of Professor Hermann Sohr, for his outstanding contributions to the mathematical theory of the Navier-Stokes equations.   \end{center}

\bigskip

\begin{abstract}  
We study the steady self-propelled motion of a rigid body immersed in an incompressible viscous fluid occupying an exterior domain in $\mathbb{R}^3$. In a body-fixed reference frame, the problem is described by a coupled fluid-rigid body system posed in a fixed exterior domain, where the Navier--Stokes equations are coupled with the unknown translational and angular velocities of the rigid body. The self-propulsion is generated by prescribing a boundary velocity on the body surface. Our main result asserts the existence of weak solutions under the sole assumption that the prescribed boundary flux is sufficiently small. The crucial step is the construction of a divergence-free extension of the boundary data that does not rely on either the zero-flux condition or the smallness of the boundary data. As a result, we generalize \cite[Theorem 5.1]{Galdi1999}, where both assumptions were required.
\end{abstract}

\maketitle


\section{Introduction}
Consider a rigid body $\mathcal S \subset {\mathbb R}^3$ moving by self-propulsion through an incompressible viscous fluid that occupies the exterior domain $\Omega = {\mathbb R}^3 \setminus \mathcal S$, as presented in \cite{GaRev}, where it is assumed that the motion of this mechanical system is described from a reference frame attached to $\mathcal S$, with origin at the center of mass of $\mathcal S$. Following \cite{puk89,Galdi1999,GaRev}, we say that a rigid body is \emph{self-propelled} if it moves without the action of external forces, the propulsion being generated solely through its interaction with the surrounding fluid. This may be realized, for instance, by prescribing inflow through some portions of the boundary $\partial\Omega$ and outflow through others, or by imposing tangential motion on parts of $\partial\Omega$, as in the case of moving belts. Typical examples include jet aircraft, submarines, and microscopic organisms such as ciliates and flagellates (see \cite{GaRev} and the references therein). The mathematical analysis of steady self-propelled motion in a viscous incompressible fluid was initiated by Galdi \cite{Galdi1999}, who proved the existence of steady weak solutions for the coupled fluid-rigid body system subject to Dirichlet boundary conditions. Later, Silvestre \cite{Sil1,MR1953783} investigated further qualitative aspects of these solutions, including the questions of attainability and symmetry. The work \cite{mmnp} extends the theory to density-dependent incompressible viscous fluids by establishing the global existence of weak solutions. In \cite{NRS1}, we study the steady self-propelled motion in a three-dimensional incompressible Navier--Stokes fluid with Navier-slip boundary conditions, proving existence of weak solutions and deriving a necessary and sufficient condition for self-propulsion. The time-dependent self-propulsion problem driven by prescribed periodic shape deformations was investigated in \cite{GP26}, where the authors established the existence of self-propelled motions for sufficiently small-amplitude oscillations and derived explicit expressions for the propulsion velocity in representative examples. The time-periodic self-propulsion problem driven by a prescribed boundary velocity was investigated in \cite{GP20}. The authors established sufficient conditions for net self-propulsion under small-amplitude boundary actuation and showed that the analysis differs fundamentally depending on whether the time average of the prescribed boundary velocity vanishes. In \cite{MTTT08}, the authors establish the global well-posedness (up to possible collisions) of a fluid--structure interaction model for self-propelled deformable solids in a two-dimensional incompressible viscous fluid and propose a numerical method for simulating the motion of multiple swimmers. In \cite{khapa1}, the authors prove the well-posedness for a bio-inspired self-propelled swimmer, represented by a chain of elastically connected rigid parts whose internal forces generate propulsion in two- and three-dimensional incompressible Navier--Stokes fluids.

The mathematical model describing the steady motion of the coupled rigid body–fluid system is formulated by the following system of equations:

\begin{equation}
\label{eqn1:NSD}
\left\{
    \begin{aligned} \displaystyle - \nabla \cdot  {\mathbb T}(v,p)+\mathrm{Re}\left[(v-V)\cdot \nabla v + \omega \times v \right] & = 0 \quad  \mbox{ in } \Omega ,\\ 
    \displaystyle  \nabla \cdot v & =0 \quad  \mbox{ in } \Omega ,\\
    v-V  &= v_{*} \quad  \mbox{ on } \partial \Omega, \\  \displaystyle \lim_{\left| x\right| \rightarrow
\infty }v(x) &=0, \\ 
 \mathrm{Re}\ m \, \omega \times \xi = -\int\limits_{\partial \Omega} \left[ \mathbb{T} (v,p) n \right. - &\left.\mathrm{Re}\ v(v-V)\cdot n \right] \dS , \\ 
 \mathrm{Re}\ \omega \times (J \omega) = -\int\limits_{\partial \Omega} x \times \left[ \mathbb{T} (v,p) n \right. - &\left.\mathrm{Re}\ v(v-V)\cdot n \right] \dS,  
   \end{aligned}
\right.
\end{equation}
where ${\mathbb T}(v,p):=2 {\mathbb D}(v) - p {\mathbb I}$, with ${\mathbb D}(v) := (\nabla v + (\nabla v)^\top)/2$ the rate-of-strain tensor, $n$ is the unit outward normal to $\partial \Omega$, and $V(x):=\xi + \omega \times x$ with $\xi \in {\mathbb R}^3$ and $\omega \in {\mathbb R}^3$. In system \eqref{eqn1:NSD}, the unknowns are the velocity $v: \Omega \to {\mathbb R}^3$ and the pressure $p: \Omega \to {\mathbb R}$ of the fluid, and the velocity $V$ of the rigid body, i.e., $\xi,\omega  \in {\mathbb R}^3$. The constant $\mathrm{Re}$ is the Reynolds number, $m$ is the mass of $\mathcal S$ and $J$ is its inertia matrix. The self-propelled motion of the body is generated by the thrust velocity $v_{*}$ prescribed on $\partial \Omega$.

The aim of this paper is to strengthen the existence result proved in \cite[Theorem 5.1]{Galdi1999}. In order to describe the problem and the main results of this paper,  in what follows, we denote the flux of $v_*$ through $\partial \Omega$ by 
\begin{equation}
\Phi:=\int_{\partial \Omega}  v_* \cdot n \dS.
\label{defflu}
\end{equation} 
The main technical ingredient in the proof of our result is the construction of a solenoidal  extension of the boundary values $v_*$ that does not require neither a null flux $\Phi$ nor a small norm $\|v_*\|_{1/2,2,\partial \Omega}$.

In order to prove existence of a weak solution, we use a classical approximation argument, the so-called invading domains technique, which reduces the problem to bounded domains \cite{Lady,Hey,Sohr2012}. Specifically, one considers a sequence  
$\{\Omega_k\}_{k=1}^{\infty}$ of bounded Lipschitz subdomains with the property
\begin{equation}
\Omega=\bigcup_{k=1}^{\infty}\Omega_k,  \qquad \Omega_k \subset \Omega_{k+1},
\label{omcad}
\end{equation}
and constructs a sequence of approximate solutions $\{v_k\}_{k=1}^{\infty}$ that converges to $v$ as $k \to \infty$ in an appropriate sense. We adopt the framework developed by Hermann Sohr \cite{Sohr2012} to construct weak solutions in the bounded domains $\Omega_k$.

Note that, if the rigid body velocity $V$ is known, the application of the Hopf method and the invading domains technique  to the Navier-Stokes problem 
\begin{equation}
\label{eqn1:NSclassical}
\left\{
    \begin{aligned} \displaystyle - \nabla \cdot  {\mathbb T}(v,p)+\mathrm{Re}\left[(v-V)\cdot \nabla v + \omega \times v \right] & = 0 \quad  \mbox{ in } \Omega ,\\ 
    \displaystyle  \nabla \cdot v & =0 \quad  \mbox{ in } \Omega ,\\
    v &= v_{*} + V \quad  \mbox{ on } \partial \Omega, \\  \displaystyle \lim_{\left| x\right| \rightarrow
\infty }v(x) &=0
   \end{aligned}
\right.
\end{equation}
yields for each $\gamma >0$, a solenoidal extension $\widetilde{v_* + V}$ of the inhomogeneous boundary data satisfying (see \cite{G})
\begin{equation}
 \mathrm{Re}  \left|  \int_{\Omega_R} u\cdot \nabla u \cdot \widetilde{(v_* + V)} \dx  \right| \leq 
\left(\gamma +  C(\Omega) \,  \mathrm{Re} \, | \Phi  | \right) \| \nabla u \|_{2,\Omega_k}^2 , \quad  \forall u \in W^{1,2}_{0,\sigma}(\Omega_k),
\label{esttrinonlcl}
\end{equation}
where $W^{1,2}_{0,\sigma}(\Omega_k)$ is the completion $\overline{C^\infty_{0,\sigma}(\Omega_k)}^{\| \nabla (\cdot) \|_{2,\Omega_k}}$ of $C^\infty_{0,\sigma}(\Omega_R):=\{ u \in C^\infty_0(\Omega_k)^3 \mid \nabla \cdot u = 0 \text{ in } \Omega_k\}$ with respect to the norm $u \mapsto \| \nabla  u \|_{2,\Omega_k}$.  In the case of \eqref{eqn1:NSclassical}, extending each element $u \in \widehat{W}^{1,2}_{0,\sigma}(\Omega_k)$ by zero to the exterior domain $\Omega$,  the trivial continuous embeddings associated to \eqref{omcad} hold
$$
W^{1,2}_{0,\sigma}(\Omega_k) \hookrightarrow W^{1,2}_{0,\sigma}(\Omega_{k+1}) \hookrightarrow \widehat{W}^{1,2}_{0,\sigma}(\Omega) \hookrightarrow L^6(\Omega)^3,
$$
where $\widehat{W}^{1,2}_{0,\sigma}(\Omega) = \overline{C^\infty_{0,\sigma}}^{\| \nabla (\cdot) \|_{2,\Omega}}$. In this way, the sequence of approximate solutions $\{v_k\}_{k=1}^{\infty}$ is such that $v_k = u_k  + \widetilde{v_* + V}$ and $\{u_k\}_{k=1}^{\infty}$  can be treated as a sequence in $\widehat{W}^{1,2}_{0,\sigma}(\Omega)$.

In the case of the coupled system \eqref{eqn1:NSD}, when seeking a weak solution, an appropriate extension of $v_*$ is crucial to control several integrals involving nonlinear terms with both the unknown velocity of the fluid and the unknown velocity of the rigid body. More precisely, in Section \ref{WFFS} (see the discussion on the nonlinear term in \eqref{expressN}) and in Section \ref{CWSG}, it will become evident that we need an estimate of the nonlinear terms in the form 
$$
2 \,  \mathrm{Re}  \left|  \int_{\Omega_R} (u - V) \cdot {\mathbb W}(u) \cdot \widetilde{v_*} \dx  \right| \leq 
\left(\gamma +  C(\Omega) \,  \mathrm{Re} \, | \Phi  | \right) \| {\mathbb D}(u) \|_{2,\Omega_R}^2 , \quad  \forall u \in {\mathcal V}_R,
$$
where $\displaystyle {\mathbb W} (v) := \left( \nabla v - (\nabla v)^\top \right) /2$ is the vorticity tensor, and the domains $\Omega_R$ and the spaces ${\mathcal V}_R$ will be introduced in Section \ref{WFFS}, when we define the functional framework for the problem \eqref{eqn1:NSD}. This framework was proposed in \cite{GS2002} when implementing a generalization of the invading domains technique that was developed in \cite{Hey} for the classical exterior Navier-Stokes equations.

In Section \ref{WFFS}, we recall the weak formulation of system \eqref{eqn1:NSD} and the functions spaces used to solve the problem. The extension of  $v_*$ to the exterior domain $\Omega$ is provided in Section \ref{sec3} (see Lemma \ref{lemaextD}). In Section \ref{CWSG}, we prove the following existence result:
\begin{theorem}
Suppose $\partial \Omega$ is locally Lipschitz continuous. Let $v_{*} \in W^{\frac{1}{2},2}(\partial \Omega)^3$ with flux \eqref{defflu} through $\partial \Omega$ satisfying: 
$$
 \mathrm{Re} \, C_0(\Omega) | \Phi  | < 1,
$$
for a specific positive constant $C_0(\Omega)$ depending only on the domain $\Omega$. Then problem \eqref{eqn1:NSD} admits at least one weak solution. If $\| v_{*} \|_{1/2,2,\partial \Omega} \leq M$, for some $M>0$, and $$
 \mathrm{Re} \, C_0(\Omega) | \Phi  | \leq 1/2,
$$
then 
\[
|\xi|+|\omega|+\|\nabla v \|_{2,\Omega} \leq C(M,\Omega)  \, \mathrm{Re} \, ( \|v_*\|_{1/2,2,\partial \Omega} + \|v_*\|^2_{1/2,2, \partial \Omega} ). 
\]
\label{maim}
\end{theorem}

We emphasize that, in \cite[Theorem 5.1]{Galdi1999}, the assumptions on the thrust velocity are: $\Phi=0$ and sufficiently small $\| v_* \|_{1/2,2,\partial \Omega}$. Here, Theorem \ref{maim} is established for a sufficiently small flux $|\Phi|$.

\section{Functional setting and Weak formulation}
\label{WFFS}

We begin by introducing the notation for the function spaces used throughout the paper. For $q \in [1,\infty]$ and a domain ${\mathcal D} \subseteq {\mathbb R}^3$, $L^q({\mathcal D})$ denotes classical Lebesgue spaces with norms $\| \cdot \|_{q,{\mathcal D}}$. For $m \in {\mathbb N}$, we consider the Sobolev spaces 
$W^{m,2}({\mathcal D})$ 
with norms $\| \cdot \|_{m,2,{\mathcal D}}$. By $W^{m-\frac{1}{2},2}(\partial{\mathcal D})$ we indicate the trace space for $W^{m,2}(\mathcal D)$-functions on the (sufficiently
smooth) boundary $\partial{\mathcal D}$ of ${\mathcal D}$, equipped with the norm $\| . \|_{m-\frac{1}{2},2,\partial{\mathcal D}}.$

In our context, we assume that $\partial \Omega$ is locally Lipschitz. We know that, by the Hardy inequality, which controls the weighted decay relative to the distance from the origin, any element $u$ of $\widehat{W}^{1,2}_{0,\sigma}(\Omega)$ satisfies
$$
\int_\Omega  \frac{|u(x)|^2}{(1+|x|)^2} \dx \leq 4 \| \nabla u \|_{2,\Omega}^2
$$
and therefore
$$
\widehat{W}^{1,2}_{0,\sigma}(\Omega) \hookrightarrow {\mathcal W}:= \left\{ u\in W^{1,2}_{loc}(\overline{\Omega})^3 \mid \frac{u}{1+|x|} \in L^2(\Omega)^3,\, \nabla u \in L^2(\Omega)^{3\times 3}, \,  \nabla \cdot  u =0 \textrm{ in } \Omega \right\}.
$$

Now, we introduce the set of rigid velocity fields,
\begin{equation*}
\mathcal{R}=\left\{ V : \mathbb{R}^3 \to \mathbb{R}^3 \mid \exists\ \xi, \omega \in \mathbb{R}^3 \text{ such that }V(x)=\xi+ \omega \times x\mbox{ for any } x\in\mathbb{R}^3\right\},
\end{equation*}
and the set of test functions to be used in the weak formulation of \eqref{eqn1:NSD},
    \[
{\mathcal T} = \{ \varphi \in C^\infty_{0,\sigma}(\overline{\Omega})^3 \mid \,  \exists \ \varphi_{\mathcal S}\in \mathcal{R} \text{ such that } \varphi  = \varphi_{\mathcal S} \text{ in a neighborhood of  } \partial \Omega   \}
\]
where $\varphi_{\mathcal S}(x) := a_\varphi + b_\varphi \times x$, for some $a_\varphi,b_\varphi  \in {\mathbb R}^3$. The completion of ${\mathcal T}$ with respect to the norm $u \mapsto \sqrt{2} \| {\mathbb D}(u)\|_{2,\Omega}$ is the space (see \cite{Serre} and \cite{Galdi1999}):
\begin{multline*}
\displaystyle {\mathcal V} = \left\{ u\in W^{1,2}_{loc}(\overline{\Omega})^3 \mid \frac{u}{1+|x|} \in L^2(\Omega)^3,\, \nabla u \in L^2(\Omega)^{3\times 3}, \,  \nabla \cdot  u =0 \textrm{ in } \Omega, \right. \\ \left. \exists  \, u_{\mathcal S}  \in {\mathcal R} \mbox{ s. t. } (u - u_{\mathcal S})|_{\partial \Omega}= 0 \right\}.
\end{multline*}
Here and in what follows, $|_{\partial \Omega}$ denotes the trace on $\partial \Omega$. By Korn inequality, we have 
\begin{equation}
\| \nabla u \|_{2,\Omega} \leq \sqrt{2} \| \mathbb{D}(u) \|_{2,\Omega} \, , \quad \forall\ u \in {\mathcal V},
\label{kor}
\end{equation}
and, from \cite{We73}, there exists a positive constant $C(\partial \Omega)$ such that
\begin{equation}
|a_u| + |b_u| \leq C(\partial \Omega) \| \mathbb{D}(u) \|_{2,\Omega}, \quad \forall\ u \in {\mathcal V},
\label{abd}
\end{equation}
where $u_{\mathcal S}(x) := a_u + b_u \times x$. By Sobolev inequality, combined with \eqref{kor}, 
\begin{equation}
\| u \|_{6,\Omega} \leq 2 \sqrt{\frac{2}{3}} \| \mathbb{D}(u) \|_{2,\Omega}, \quad \forall\ u \in {\mathcal V}.
\label{n6d}
\end{equation}

  \begin{definition}
    A pair $(v,V) \in {\mathcal W} \times \mathcal{R}$, with $V(x)=\xi + \omega \times x$, is a weak solution in the velocity component of the system \eqref{eqn1:NSD} if
\[
\begin{aligned}
2\int_{\Omega} {\mathbb D}({v}): {\mathbb D}({\varphi}) \dx 
=  & \,  {\mathrm{Re}} \int_{\Omega}(v-V) \cdot \nabla \varphi \cdot v \dx - {\mathrm{Re}} \int_{\Omega} ( \omega \times v ) \cdot \varphi \dx \\
 & + {\mathrm{Re}} \, m \, \xi\times\omega \cdot a_\varphi  + {\mathrm{Re}}\, (J \omega) \times\omega \cdot b_\varphi, \quad \forall\  \varphi \in {\mathcal T}, 
 \end{aligned}
 \]
 and  $$(v-V)|_{\partial \Omega}  = v_* \, .$$
   \end{definition}  Following the usual ansatz $v = u + \widetilde{v_*}$, where $\widetilde{v_*}$ is a suitable solenoidal extension of $v_*$, the weak formulation for the velocity components $(u,V) \in {\mathcal V} \times {\mathcal R}$ can be rewritten as:
\begin{equation}
\begin{aligned}
 & 2\int_{\Omega} {\mathbb D}({u}): {\mathbb D}({\varphi})\dx 
 =   \mathrm{Re} \int_{\Omega}(u-V) \cdot \nabla \varphi \cdot u \dx -  \mathrm{Re}  \int_{\Omega} ( \omega \times u ) \cdot \varphi \dx \\
 & + \mathrm{Re} \, m \, \xi\times\omega \cdot a_\varphi  + \mathrm{Re} \, (J\omega)\times\omega \cdot b_\varphi +  \mathrm{Re} \int_{\Omega} \widetilde{v_*} \cdot \nabla \varphi \cdot u \dx+ \mathrm{Re} \int_{\Omega}(u-V) \cdot \nabla \varphi \cdot \widetilde{v_*} \dx \\
 & - \mathrm{Re} \int_{\Omega} ( \omega \times \widetilde{v_*}) \cdot \varphi \dx 
  + \mathrm{Re} \int_{\Omega} \widetilde{v_*} \cdot \nabla \varphi \cdot \widetilde{v_*} \dx - 2\int_{\Omega} {\mathbb D}(\widetilde{v_*}): {\mathbb D}({\varphi})\dx , \quad \forall\  \varphi \in {\mathcal T}.
\end{aligned}
\label{wform}
\end{equation}
Note that the relation between $u$ and $V(x)= \xi + \omega \times x$ is 
\begin{equation}
 (u-V)|_{\partial \Omega} = (u - u_{\mathcal S})|_{\partial \Omega} = 0,
 \label{uVuS}
 \end{equation}  
which is automatically fulfilled in the space ${\mathcal V}$. To simplify the notation, we will henceforth use 
$u_{\mathcal S}$ and $V$ interchangeably.

To address the weak formulation \eqref{wform} for the unknown $(u,\xi,\omega)$, or simply $u$, if implicitly we assume that $u$ and $V$ are interrelated by \eqref{uVuS}, we first introduce some additional notation. As usual, $B_R=\{ x \in {\mathbb R}^3 \mid |x| <R \}$ stands for the Euclidean ball of radius $R>0$. We will also use the notation $B^R:=\{ x \in {\mathbb R}^3 \mid |x| > R \}$. We denote by $u|_{\partial B_R}$ the trace of $u \in W^{1,2}(\Omega_R)^3$ on $\partial B_R$. Denoting by $\delta(\mathcal{S})$, the diameter of $\mathcal{S}$, for $R>\delta(\mathcal{S})$, we define the bounded domains $\Omega_R := \Omega \cap B_R$, and the function spaces 
\begin{equation*}
{\mathcal V}_R=  \left\{ u \in W^{1,2}(\Omega_R)^3 \mid \nabla \cdot  u =0 \textrm{ in }\Omega_R, \, \exists \ u_{\mathcal S}  \in {\mathcal R} \mbox{ s. t. } (u - u_{\mathcal S})|_{\partial\Omega}=0, \text{ and } u|_{\partial B_R}= 0  \right\},
\end{equation*}
equipped with inner product  and norm
$$(u,v)_{{\mathcal V}_R} := 2 \int_{\Omega_R} {\mathbb D}(u):{\mathbb D}(v) \dx , \qquad \|u\|_{\mc{V}_R} := \sqrt{2}\| {\mathbb D}(u)\|_{2,\Omega_R}.$$ Extending each element $u \in {\mathcal V}_R$ by zero to the exterior domain $\Omega$, we obtain ${\mathcal V}_R \subset {\mathcal V}$ and the continuous embeddings
$$
{\mathcal V}_R  \hookrightarrow  {\mathcal V}_{R'} \hookrightarrow {\mathcal V} \hookrightarrow  L^6(\Omega)^3 \qquad (R' > R).
$$

We also define the Hilbert spaces 
\begin{multline*}
    {\mathcal{H}}_R = \Big\{ u \in L^2(\Omega_R) \mid \nabla \cdot  u =0 \textrm{ in }\Omega,  \exists \, u_{\mathcal S} \in {\mathcal R} \mbox{ such that }
  {(u - u_{\mathcal S})\cdot n}_{|\partial \Omega} = 0, \\ \text{ and } u\cdot n|_{\partial B_R}= 0\Big\}
\end{multline*}
with the scalar product 
$$ 
(u,v)_{{\mathcal{H}}_R}  :=  \int_{\Omega} u \cdot v \dx + m a_u \cdot a_v + b_u \cdot J  b_v \, .
$$ 
The induced norm is 
\begin{equation}
\| u  \|_{{\mathcal{H}}_R} :=  \left( \| u \|_{2,\Omega_R}^2 + m |a_u|^2 + b_u \cdot J b_u \right)^{\frac 12}.
\label{normHR}
\end{equation}
Then the following embeddings hold
\begin{equation}
{\mathcal V}_R  \hookrightarrow  \hookrightarrow  {\mathcal H}_{R} \hookrightarrow {\mathcal V}'_R 
\label{compin}
\end{equation}
where the first embedding is compact and the second is continuous.

The Stokes operator plays a fundamental role in the functional analytic treatment of the classical Navier--Stokes equations, as developed, for instance, by Hermann Sohr in \cite{Sohr2012}. To carry out our analysis within a similar functional analytic framework, we now introduce a generalization of the Stokes operator adapted to our problem on a bounded domain. Specifically, we consider the following system
\begin{equation}
\left\{
\begin{aligned}
- \nabla \cdot  {\mathbb T}(u,p) & =  f   && \text{ in }\Omega_R  \\ 
\nabla \cdot u & =0  && \text{ in } \Omega_R  \\  
u & = V  && \text{ on }\partial \Omega \\ 
u  & = 0 && \text{ on } \partial B_R   \\ 
 \int_{\partial \Omega}  {\mathbb T}(u,p) n \dS  & =  m \, a_f  &&  \\ 
 \int_{\partial \Omega} x\times  {\mathbb T}(u,p) n \dS  & =   J b_f . &&
\end{aligned}
\label{systoke}
\right.
\end{equation}
Suppose $f \in {\mathcal{H}}_R$ with $f_{\mathcal S}:= a_f  +  b_f \times x$. Then the weak formulation of \eqref{systoke} in ${\mathcal{V}}_R$  is 
\[
\begin{aligned}
2 \int_{\Omega_R} {\mathbb D}(u):{\mathbb D}(v) \dx  &  =   \int_{\Omega}  f \cdot v \dx + m \, a_f  \cdot a_v + b_f \cdot (J b_v), \quad \forall v \in {\mathcal{V}}_R \\
\iff \quad  (u,v)_{{\mathcal{V}}_R}  & =  (f,v)_{{\mathcal{H}}_R}, \quad \forall  v \in {\mathcal{V}}_R .
\end{aligned}
\]

We now adapt the arguments of \cite[Chapter~III, Section~2]{Sohr2012}, to our functional setting. Let  $D(A) \subseteq {\mathcal{V}}_R$  be the space of all those $u \in {\mathcal{V}}_R$ for which there exists some $f \in {\mathcal{H}}_R$  satisfying
\begin{equation}
 (u,v)_{{\mathcal{V}}_R}  =  (f,v)_{{\mathcal{H}}_R}, \quad \forall\ v \in {\mathcal{V}}_R.
\label{rie}
\end{equation}
By the Riesz representation theorem,  $D(A)$ is the space of all those $u \in {\mathcal{V}}_R$  such that the functional
$$
{\mathcal{V}}_R \ni v  \longmapsto 2 \int_{\Omega_R} {\mathbb D}(u):{\mathbb D}(v) \dx
$$
is continuous in the norm \eqref{normHR}. For all $u \in D(A)$, let $Au \in {\mathcal{H}}_R$ be defined by the relation
\begin{equation}
(Au,v)_{{\mathcal{H}}_R} = (u,v)_{{\mathcal{V}}_R}  , \quad \forall v \in {\mathcal{V}}_R  
\label{defA}
\end{equation}
$$
\iff \int_{\Omega_R}  Au  \cdot v \dx + m a_{Au}  \cdot a_v + b_{Au} \cdot J  b_v = 2 \int_{\Omega_R} {\mathbb D}(u):{\mathbb D}(v) \dx , \quad \forall v \in {\mathcal{V}}_R.
$$
Thus $Au = f$ as in \eqref{rie} and we call $A_R = A$ the generalized Stokes operator for the domain $\Omega_R$. 

If $\partial \Omega$ is of class $C^2$, then, for every $u \in D(A)$, there exists $p \in W^{1,2}(\Omega_R)$ such that $Au \in {\mathcal{H}}_R$ is characterized by (see \cite{GS2002} for details):  
$$
\begin{cases}  \displaystyle Au  = - \Delta u + \nabla p  \text{ in } \Omega_R ,\\   \displaystyle  ( Au )_{\mathcal S} =  \frac{1}{m} \int_{\partial \Omega}  (2 {\mathbb D}(u) - p {\mathbb I} )n \dS  + \left( J^{-1} \int_{\partial \Omega}  y \times ( 2 {\mathbb D}(u) - p {\mathbb I} ) n \dS \right) \times x . \end{cases}
$$

Let 
$$
{\mathcal{U}}_R = \left\{ \varphi \in W^{1,2}(\Omega_R)^3 \mid  \exists \ \varphi_{\mathcal S}  \in {\mathcal R} \mbox{ s. t. } (\varphi - \varphi_{\mathcal S})|_{\partial\Omega}=0, \text{ and } \varphi|_{\partial B_R}= 0  \right\}.
$$
If $\partial \Omega$ is Lipschitz, then we can still recover the pressure from \eqref{rie}. Since $C^\infty_{0,\sigma}(\Omega_R)^3 \subset {\mathcal{V}}_R$ we obtain
$$
2 \int_{\Omega_R} {\mathbb D}(u):{\mathbb D}(v) \dx   =   \int_{\Omega}  f \cdot v \dx, \quad \forall v \in C^\infty_{0,\sigma}(\Omega_R)^3.
$$
It follows that there exists $p \in L^2(\Omega_R)$ such that 
\begin{equation}
- \nabla \cdot {\mathbb T}(u,p) = f \text{ in } \Omega_R
\label{clm}
\end{equation}
in the sense of of distributions. Since $f \in L^2(\Omega_R)^3$ and ${\mathbb T}(u,p) \in L^2(\Omega_R)^{3 \times 3}$, we can test \eqref{clm} with $\varphi \in {\mathcal{U}}_R$, which yields
$$
\langle  {\mathbb T}(u,p)n ,  \varphi_{\mathcal S} \rangle_{W^{-1/2,2}(\partial \Omega)} = \int_{\Omega_R} \left( 2 {\mathbb D}(u) - p \mathbb I \right) : \nabla \varphi \dx -  \int_{\Omega_R }f \cdot  \varphi \dx.
$$
Now choosing $\varphi  \in {\mathcal{V}}_R$, with $\varphi_{\mathcal S} = e_i$, $i=1,2,3$, and comparing with \eqref{rie} we conclude that 
$$
\langle  {\mathbb T}(u,p)n , 1 \rangle_{W^{-1/2,2}(\partial \Omega)\times W^{1/2,2} (\partial \Omega)}  = m a_f.
$$
In a similar way, choosing $\varphi_{\mathcal S} = e_i \times x$, $i=1,2,3$ yields
$$
\langle x \times  {\mathbb T}(u,p)n , 1 \rangle_{W^{-1/2,2}(\partial \Omega) \times W^{1/2,2} (\partial \Omega) }  = J b_f.
$$

These results can be summarized in the following way.
\begin{lemma}
Let $u \in {\mathcal{V}}_R$ and $f \in {\mathcal{H}}_R$. Then $u$ is a weak solution of the Stokes
system \eqref{systoke} if and only if $u \in D(A)$ and $Au = f$ 
and this holds  if and only if there exists a unique $p \in L^2(\Omega_R)$ satisfying \eqref{systoke} in a generalized sense.
\end{lemma}

In each bounded domain $\Omega_R$, from \eqref{defA} we have
$$
u \in {\mathcal V}_R \text{ and }Au = 0 \iff u= 0
$$
and therefore, $N(A)=\{0\}$. Consequently,  $A^{-1}$ is a bounded operator with
$$
D(A^{-1}) = R(A) = {\mathcal H}_R.
$$

The square root of the Stokes operator and its inverse play an
important role in the theory of weak solutions of the Navier-Stokes system, as shown in \cite{Sohr2012}. In our case, for the generalized Stokes operator defined by \eqref{defA}, the operator $A^{\frac 12}$ has the following properties:
$$
D(A^{\frac 12}) = {\mathcal{V}}_R, \qquad  R(A^{\frac 12}) = {\mathcal{H}}_R, 
$$
\begin{equation}
 (A^{\frac 12} u , A^{\frac 12} v)_{{\mathcal{H}}_R} = (u,v)_{{\mathcal{V}}_R} = 2 \int_{\Omega_R} {\mathbb D}(u):{\mathbb D}(v) \dx \qquad (u,v \in {\mathcal{V}}_R) . 
 \label{a12}
\end{equation}
In particular, it follows from \eqref{a12} that
\begin{equation}
\| A^{\frac 12} u \|_{{\mathcal{H}}_R} = \|u \|_{{\mathcal{V}}_R} \qquad (u \in {\mathcal{V}}_R).
\label{na12}
\end{equation}
Moreover 
$$
D(A^{-1/2}) = R(A^{1/2}) = {\mathcal H}_R
$$
and $A^{-1/2}$ is a bounded operator. Actually, $A^{-1/2}$ is a compact operator, as a consequence of 
\begin{equation}\label{A-12}
\| U \|_{{\mathcal{H}}_R} = \| A^{-\frac 12} U \|_{{\mathcal{V}}_R} \qquad (U \in {\mathcal{H}}_R)
\end{equation}
together with the compact embedding in \eqref{compin}.

In Section \ref{CWSG}, a weak solution will be constructed using the invading domains technique  \cite{Hey,GS2002} and a procedure similar to the one used in \cite[pg. 168]{Sohr2012} for the classical homogeneous exterior Navier-Stokes problem. Here, a  key step is the construction of an appropriate solenoidal extension of $v_*$. Then, for each  $R >  \delta({\mathcal S})$, an approximate solution $(u_R,\xi_R,\omega_R) \in {\mathcal V}_R \times {\mathbb R}^3 \times {\mathbb R}^3$ will be obtained in the bounded domain $\Omega_R$ using the Leray-Schauder principle. A solution $(u,V) \in {\mathcal V} \times {\mathcal R}$ is obtained upon taking the limit as $R \to \infty$ by virtue of uniform estimates with respect to $R$. 

It should be emphasized that the main difficulty inherent to problem \eqref{wform} when restricted to a bounded domain $\Omega_R$ consists precisely in deriving a uniform bound in $R$ for the integrals containing the nonlinear terms, specifically,
\[
N:=  \int_{\Omega_R} \widetilde{v_*} \cdot \nabla u_R \cdot u_R \dx + \int_{\Omega_R}(u_R-V_R) \cdot \nabla u_R \cdot \widetilde{v_*} \dx  - \int_{\Omega_R} ( \omega_R \times \widetilde{v_*}) \cdot u_R \dx.
\]
Using the relation 
\[ \omega_R \times \widetilde{v_*} = \widetilde{v_*} \cdot \nabla V_R,
\]
and integration by parts together with the properties of the trace $u_{R}|_{\partial \Omega_R}$, namely
$$
u_{R}|_{\partial \Omega} = V_R, \quad u_{R}|_{\partial B_R} = 0,
$$
we can write a simplified expression for $N$:
\begin{align}\label{expressN}
N = & \int_{\Omega_R} \widetilde{v_*} \cdot \nabla (u_R - V_R) \cdot u_R \dx + \int_{\Omega_R}(u_R-V_R) \cdot \nabla u_R \cdot \widetilde{v_*} \dx \notag
\\
= & - \int_{\Omega_R} \widetilde{v_*} \cdot \nabla u_R \cdot (u_R - V_R) \dx + \int_{\Omega_R}(u_R-V_R) \cdot \nabla u_R \cdot \widetilde{v_*} \dx \notag\\
= & \,  2 \int_{\Omega_R}(u_R-V_R) \cdot {\mathbb W} (u_R) \cdot \widetilde{v_*} \dx.
\end{align}
Note that the structure of this integral is considerably different from that of the integral in \eqref{esttrinonlcl}. Here, the difficulty is to handle the traces of $u_R -V_R$ on $\partial \Omega_R$. While  $(u_R-V_R)|_{\partial \Omega}=0$, we have $(u_R-V_R)|_{\partial B_R}\not=0$, which prevents directly applying the Hardy inequality near the boundary $\partial B_R$ as done in the classical Dirichlet problem for the Navier-Stokes equations (see \cite[Lemma X.4.2]{G}).

 In the next section, we construct a solenoidal extension $\widetilde{v_*}$ that allows for an appropriate bound on $R$ of the final integral used to express $N$. 

\section{Extension of the boundary values}\label{sec3}

 Recall the Laplace fundamental solution ${\mathcal E} (x) := \frac{1}{4 \pi |x|}$, a tempered distribution that satisfies $ -\Delta {\mathcal E}  = \delta_0$ in ${\mathbb R}^3$, where $\delta_a$ denotes the Dirac delta distribution with support at the point $a \in {\mathbb R}^3$. 
 
The improvement of the existence result in \cite[Theorem 5.1]{Galdi1999} stems from the possibility of constructing the following extension of $v_*$, which generalizes \cite[Lemma 5.4]{Galdi1999}.

\begin{lemma}
\label{lemaextD}
Let $\partial \Omega$ be locally Lipschitz and $v_{*} \in W^{\frac{1}{2},2}(\partial \Omega)^3$. Then for any $\gamma >0$,  there exists $\widetilde{v_*}  \in W^{1,2}( \Omega)^3$ satisfying 
\[
\begin{cases}
  \nabla \cdot \widetilde{v_*}  =  0  \text{ in } \Omega, \\
  \widetilde{v_*}  =  v_*  \text{ on } \partial \Omega.
\end{cases}
\]
Furthermore, we have the estimate
$$
2 \,  \mathrm{Re}  \left|  \int_{\Omega_R} (u - V) \cdot {\mathbb W}(u) \cdot \widetilde{v_*} \dx  \right| \leq 
\left(\gamma +  C(\Omega) \,  \mathrm{Re} \, | \Phi  | \right) \| u \|_{{\mathcal V}_R}^2 , \quad  \forall u \in {\mathcal V}_R,
$$
where $C(\Omega)$ is a positive constant that depends only on the domain $\Omega$. If $\| v_{*} \|_{1/2,2,\partial \Omega} \leq M$, for some $M>0$, then there is $C(\Omega,\gamma,M)>0$ such that
\[
\| \widetilde{v_{*}} \|_{1,2,\Omega} \leq C(\Omega,\gamma,M)\| v_{*} \|_{1/2,2,\partial \Omega}.
\]
\end{lemma}
\begin{proof}
\textbf{Step 1:} Construction of flux-carrier.

In order to lift the boundary velocity $v_*$ when $\Phi \not=0$, (recall that $\Phi$ is defined in \eqref{defflu}). Fix $x_0 \in \text{int}({\mathcal B}_0)$ and take the flux carrier
$$
\Phi \sigma (x) = \Phi \nabla {\mathcal E}(x-x_0) = - \Phi \frac{x-x_0}{4 \pi |x-x_0|^3}.
$$
Note that it is not always possible to choose $x_0 = 0$, the center of mass of the rigid body, as $0$ may not lie in $\text{int}({\mathcal B}_0)$.

Since $\nabla \cdot \sigma = \Delta {\mathcal E}(\cdot -x_0) = - \delta_{x_0}$ in  ${\mathbb R}^3$ and $x_0 \in \text{int}({\mathcal B}_0)$, we have $\nabla \cdot \left(  \Phi \sigma \right) = 0$ in $\Omega$ and therefore $\int_{\partial \Omega} \sigma (x) \cdot n \dS = 1$.

By direct calculations, one finds 
\begin{equation}
\sigma \in L^q(\Omega)^3, \, q>3/2 \text{ and } \nabla \sigma \in L^q(\Omega)^{3 \times 3}, \, q>1,
\label{sumsig}
\end{equation}
and
\begin{equation}
(\omega \times (x-x_0))  \cdot \nabla  \sigma(x) -  \omega \times \sigma(x) = 0 .
\label{divsig}
\end{equation}
     
\textbf{Step 2:} Construction of a compactly supported component of $\widetilde{v_*}$.\\
Let us define $\beta_{*} \in W^{\frac{1}{2},2}(\partial \Omega)^3$ by
$\beta_{*}:=  v_{*}  -  \Phi \sigma|_{\partial \Omega}$. As noted above, we have the following  
 \begin{equation}\label{lem4.3-1}
 \int_{\partial \Omega} \beta_{*} \cdot n \dS = 0 \text{ and } \nabla \cdot \left(  \Phi \sigma \right) =0 \text{ in }\Omega.
 \end{equation} 
In this step, our aim is to construct an extension $\widetilde{\beta_{*}}$ of $\beta_{*}$ with compact support. For fixed $R_0 > 3 \delta(\mathcal{S})$, consider $$\vartheta (x) = \begin{cases} \vartheta^{(1)} (x), \quad x \in {\mathcal S} \\  \vartheta^{(2)} (x) , \quad x \in \Omega_{R_0}  \\  0, \quad x \in \overline{B^{R_0}}, \end{cases}$$
where $ \vartheta^{(1)}  \in W^{1,2}(\text{int}({\mathcal S}))$ is a solenoidal extension of $\beta_*$ in $\mc{S}$ such that 
\begin{equation}
\| \vartheta^{(1)} \|_{1,2,\text{int}({\mathcal S})} \leq C_1 (\mathcal{S}) \| \beta_* \|_{1/2,2,\partial \Omega},
\label{v1e}
\end{equation}
and $\vartheta^{(2)} \in W^{1,2}(\Omega_{R_0})$ is a solenoidal extension of $\beta_*$ in $\Omega_{R_0}$ with $\vartheta^{(2)} |_{\partial B_{R_0}}=0$, along with the estimate 
\begin{equation}
\| \vartheta^{(2)} \|_{1,2,{\Omega_{R_0}}} \leq C_2 (\Omega_{R_0}) \| \beta_* \|_{1/2,2,\partial \Omega}.
\label{v2e}
\end{equation}
Then $\vartheta \in W^{1,2}({\mathbb R}^3)$ and $\nabla \cdot \vartheta  = 0 $ in ${\mathbb R}^3$. From \eqref{v1e} and \eqref{v2e}, we obtain
\[
\| \vartheta \|_{1,2,\Omega_{R_0}} \leq C (\Omega_{R_0}) \| \beta_* \|_{1/2,2,\partial \Omega}.
\]

Since $\vartheta$ is of compact support, the convolution $z : = {\mathcal E} * \vartheta $ is well defined and $- \Delta z = \vartheta $ in ${\mathbb R}^3$. We have $z \in W^{2,2}({\mathbb R}^3)$, and using the relations
$\Delta z = \nabla (\nabla \cdot z ) - \nabla \times (\nabla \times z)$ and
$\nabla \cdot z = {\mathcal E} * (\nabla \cdot \vartheta ) = 0$, we conclude that $\nabla \times (\nabla \times z) = \vartheta$.

Now, consider a cut-off function $\theta_{R_0} \in C^\infty_0({\mathbb R}^3)$ such that
\[
\theta_{R_0} (x) = 1, \, |x| \leq R_0/3, \qquad \theta_{R_0} (x) = 0, \, |x| \geq 2 R_0/3, \qquad  0 \leq \theta_{R_0} (x) \leq 1, \, \forall x \in {\mathbb R}^3,
\]
and take $w:= \nabla \times ( \theta_{R_0} z)$. As  the solenoidal function  $\nabla \times w|_{\Omega} $ is equal to $\vartheta^{(2)}$ in a neighborhood of $\partial \Omega$ and equal to zero in a neighborhood of $\partial B_{R_0}$, we have
\begin{equation}
\nabla \times w =  \beta_*  \text{ on } \partial \Omega, \qquad w = \nabla \times w = 0 \text{ on } \partial B_{R_0},\qquad w=0 \text{ in } B^{R_0},
\label{weq0}
\end{equation}
\begin{equation}
\| w \|_{2,2,\Omega} \leq C(\Omega_{R_0}) \|\beta_*\|_{1/2,2,\partial \Omega} \leq C(\partial \Omega, R_0) \|v_*\|_{1/2,2,\partial \Omega}.
\label{nw2}
\end{equation}

Let $0 < \varepsilon <1$ and $\Phi_\varepsilon \in C^\infty(\mathbb R)$
be such that $\Phi_\varepsilon(t) = 1$ for $t <  \frac{\exp(-2/\varepsilon)}{2}$, $\Phi_\varepsilon(t) = 0$ for $t \geq 2 \exp(-1/\varepsilon)$, $|\Phi_\varepsilon(t)| \leq 1$ and $|\Phi'_\varepsilon(t)| \leq \varepsilon/ t $, for all $t >0$. Define $d(x):=\text{dist}(x,\partial \Omega_{R_0})$, the distance of a point $x \in \Omega_{R_0}$ to the boundary $\partial \Omega_{R_0}$, and let $\rho(x)$ be the corresponding regularized distance in the sense of Stein. Using these, define the cut-off function for the domain $\Omega_{R_0}$ 
$$
\psi_\varepsilon(x) := \Phi_\varepsilon(\rho(x)), \, x \in \Omega_{R_0}
$$ 
and extend it by 1 to the exterior domain $\Omega^{R_0}:= \Omega \cap B^{R_0}$. The cut-off function satisfies 
\begin{equation}\label{psiep}
\psi_\varepsilon(x)= \begin{cases}
1 & \text{if} \quad d(x) < \frac{\exp(-2/\varepsilon)}{2 \kappa_1},\\
0 & \text{if} \quad d(x) \geq 2 \exp(-1/\varepsilon),\\
1 & \text{if}\quad x\in \Omega^{R_0},
\end{cases}
\end{equation}
and we also have (see \cite[Lemma III.6.2]{G})
\begin{equation}
|\nabla \rho(x)| \leq \kappa_2 , \qquad \nabla \psi_\varepsilon(x) = \Phi'_\varepsilon(\rho(x)) \nabla \rho(x),   \qquad | \nabla \psi_\varepsilon(x) | \leq \frac{\varepsilon \kappa_2}{d(x)},
\label{psie12}
\end{equation}
where $\kappa_1$ and $\kappa_2$ are positive constants independent of the domain. Furthermore,  given $s>0$, there exists a positive constant $c=c(s,\Omega,R_0)$ such that (see \cite[Lemma IX.4.2]{G})
\begin{equation}
|\text{supp} (\psi_\varepsilon)|^s \leq c \varepsilon.
\label{eqep3}
\end{equation}

For a fixed $\varepsilon >0$, the desired extension of $\beta_*$ is defined by $\widetilde{\beta_*} = \nabla \times (\psi_\varepsilon w)$, with $\psi_\varepsilon$ defined in \eqref{psiep}.

\textbf{Step 3:} Definition of extension of $v_*$ and its properties.

 Let us define the  extension of $v_*$ by
\[
\widetilde{v_*} (x) =  \nabla \times (w(x)\psi_\varepsilon(x))  + \Phi \sigma(x) 
 = \nabla \psi_\varepsilon(x) \times w(x) +  \psi_\varepsilon(x) \nabla \times w(x)   + \Phi \sigma(x) , \quad  x \in \Omega.
\]
 The choice of $\varepsilon$ will depend on the given $\gamma$.

For $R > R_0$, we have
\[
\begin{aligned}
\int_{\Omega_R}  (u  - V) \cdot {\mathbb W}(u) \cdot \widetilde{v_*} \dx 
 =  & \int_{\Omega_{R_0}} \psi_\varepsilon  (u - V) \cdot {\mathbb W}(u) \cdot  \nabla \times w \dx \\
& + \int_{\Omega_{R_0}} (u - V) \cdot {\mathbb W}(u) \cdot ( \nabla  \psi_\varepsilon \times w )\dx  \\
& +  \, \Phi   \int_{\Omega_R} (u - V)  \cdot {\mathbb W}(u) \cdot \sigma \dx =: I_1 + I_2 + I_3,
\end{aligned}
\]
where $u \in {\mathcal V}_R$ and $V=u_{\mathcal S} \in {\mathcal R}$.

Let $\Omega_{R_0,\varepsilon} := \{ x \in \Omega_{R_0} \mid d(x) < 2 \exp(-1/\varepsilon) \}$. The integral $I_1$, is estimated as
\[
|I_1| \leq  \| u - V \|_{4,\Omega_{R_0,\varepsilon}}\|{\mathbb W}(u) \|_{2,\Omega_{R_0}} \| \nabla \times w \|_{4,\Omega_{R_0}} \leq C(\Omega_{R_0})\| u - V \|_{4,\Omega_{R_0,\varepsilon}}\|u \|_{{\mathcal V}_R} \| \beta_* \|_{1/2,2,\partial \Omega},
\]
where \eqref{kor} and \eqref{nw2} have been  applied. Now, using \eqref{eqep3}, we obtain
\[
 \| u - V \|_{4,\Omega_{R_0,\varepsilon}} \leq |\Omega_{R_0,\varepsilon}|^{1/2} \| u - V \|_{6,\Omega_{R_0,\varepsilon}} \leq c \varepsilon \| u - V \|_{6,\Omega_{R_0}}
\]
and invoking \eqref{abd} and \eqref{n6d}, 
\[
 \| u - V \|_{6,\Omega_{R_0}} 
 \leq  \| u \|_{6,\Omega_{R_0}} + C(\Omega_{R_0}) \left( |\xi| + |\omega| \right) \leq C(\Omega_{R_0}) \| u \|_{{\mathcal V}_R}.
\]
Given $\gamma >0$, we can take $\varepsilon$ small enough to get (see \cite[Lemma X.4.2]{G})
\[
2 \,  \mathrm{Re} \, |I_1|  
 \leq  \frac{\gamma}{2} \| u \|_{{\mathcal V}_R}^2 .
\]    

Concerning the integral $I_2$, note that $
(u - V) \cdot {\mathbb W}(u) \cdot ( \nabla  \psi_\varepsilon \times w ) =  {\mathbb W}(u) : (u - V) \otimes ( \nabla  \psi_\varepsilon  \times w )$
where $(u-V)|_{\partial \Omega} = 0$, and recalling \eqref{weq0}, we have $w|_{\partial B_{R_0}} = 0$. Therefore,  the combined term $(u - V) \otimes  w $, satisfies
\begin{equation}
(u - V) \otimes  w  \in W^{1,2}_0(\Omega_{R_0})^{3 \times 3}.
\label{w1o}
\end{equation}
Using the properties \eqref{psie12} of $\psi_\varepsilon$ and \eqref{nw2}, it is possible to estimate $I_2$ in the form
\[
\begin{aligned}
|I_2|  \leq & \, \| {\mathbb W}(u) \|_{2,\Omega_R} \left \| (u - V) \otimes w \otimes ( \nabla  \psi_\varepsilon ) \right \|_{2,\Omega_{R_0}} \\
 \leq & \, \varepsilon C  \| u \|_{{\mathcal V}_R} \left \| \frac{(u - V) \otimes w }{d(x)} \right \|_{2,\Omega_{R_0}} \\
  \leq & \, \varepsilon C  \| u \|_{{\mathcal V}_R} \left( \left \| \nabla(u - V)  \right \|_{2,\Omega_{R_0}} \left \|  w  \right \|_{\infty,\Omega_{R_0}}
  + \left \| u - V \right \|_{6,\Omega_{R_0}} \left \| \nabla  w  \right \|_{3,\Omega_{R_0}} \right) \\
  \leq  & \, \varepsilon C(\Omega_{R_0}) \| \beta_* \|_{1/2,2,\partial \Omega}  \|u \|^2_{{\mathcal V}_R} 
\end{aligned}
\]
where the Hardy inequality has been applied to the whole $(u - V) \otimes  w $. Note that \eqref{w1o} is crucial to obtain this bound for $I_2$.
For the given $\gamma$, we choose $\varepsilon$ sufficiently small to obtain, similarly to the estimate obtained in \cite[Lemma X.4.2]{G}, 
\[
2 \,  \mathrm{Re} \, | I_2 | \leq  2 \,  \mathrm{Re}  \| {\mathbb W}(u) \|_{2,\Omega_R} \left \| (u - V) \otimes w \otimes ( \nabla  \psi_\varepsilon ) \right \|_{2,\Omega_{R_0}} \leq  \frac{\gamma}{2} \| u \|_{{\mathcal V}_R}^2.
\]

After performing integration by parts and applying relation \eqref{divsig}, the integral $I_3$ can be written as
\[
\begin{aligned} 
I_3 
 =& \,
 \frac 12  \int_{\Omega_R} (u - V)  \cdot \nabla u  \cdot \sigma \dx -  \frac 12  \int_{\Omega_R} \sigma \cdot  \nabla u \cdot (u - V)  \dx\\
 =&
  - \frac 12   \int_{\Omega_R} u \cdot \nabla  \sigma \cdot u \dx  + \frac 12   \int_{\Omega_R}   \sigma \cdot \nabla u \cdot u \dx \\
& + \frac 12   \int_{\Omega_R} \xi \cdot \nabla  \sigma \cdot u \dx + \frac 12   \int_{\Omega_R} [ (\omega \times x)  \cdot \nabla  \sigma -  \sigma \cdot \nabla (\omega \times x) ] \cdot u \dx \\
= & - \frac 12   \int_{\Omega_R} u \cdot \nabla  \sigma \cdot u \dx  + \frac 12   \int_{\Omega_R}   \sigma \cdot \nabla u \cdot u \dx + \frac 12   \int_{\Omega_R}V(x_0) \cdot \nabla  \sigma \cdot u \dx.
 \end{aligned}
\]
By resorting to \eqref{kor}, \eqref{abd} and \eqref{sumsig}, we get
\[
\begin{aligned}
2 \, \mathrm{Re}  \, |I_3| \leq & \, \mathrm{Re} \, \Phi \left(  \| u \|_{6,\Omega_R}^2 \| \nabla \sigma \|_{3/2,\Omega} + \| \sigma \|_{3,\Omega}  \| \nabla u \|_{2,\Omega_R}\| u \|_{6,\Omega_R}  \right. \medskip \\
& \left. 
+ \,   |\xi + \omega \times x_0| \| \nabla  \sigma \|_{6/5,\Omega} \| u \|_{6,\Omega_R} \right) \medskip \\
\leq & \, \mathrm{Re} \, \Phi  \, C_0(\Omega) \| u \|_{{\mathcal V}_R}^2.  
\end{aligned}
\]

Concerning the estimate for $\| \widetilde{v_*}\|_{1,2,\Omega}$, return to 
\[
\widetilde{v_*} (x)  
 = \nabla \psi_\varepsilon(x) \times w(x) +  \psi_\varepsilon(x) \nabla \times w(x)   + \Phi \sigma(x) , \quad  x \in \Omega,
\]
where $w$ and $\psi_{\varepsilon}$ satisfy \eqref{nw2} and \eqref{psiep}, respectively, and follow the same argument used in the last part of the proof of \cite[Lemma X.4.2]{G}. If $\| v_{*} \|_{1/2,2,\partial \Omega} \leq M$, for some $M>0$, then there are $\gamma >0$ and $C(\Omega,\gamma,M)>0$ such that
\[
\| \widetilde{v_{*}} \|_{1,2,\Omega} \leq C(\Omega,\gamma,M)\| v_{*} \|_{1/2,2,\partial \Omega}.
\]
\end{proof}
\begin{remark}
In the proof of Lemma \ref{lemaextD}, we have used  different cut-offs $\theta_{R_0}$, $\Phi_\varepsilon$ and $\psi_{\varepsilon}$. Here $\theta_{R_0}$ is used to obtain the boundary condition $w=0 \text{ in } B^{R_0}$ in
\eqref{weq0} which is crucial to obtain \eqref{w1o}, whereas $\Phi_\varepsilon$ and $\psi_{\varepsilon}$ are the standard cut-offs used in the Hopf method as mentioned in \eqref{esttrinonlcl}.
\end{remark}
\section{Existence of a weak solution}
\label{CWSG}
The extension of $v_*$ constructed in Lemma \ref{lemaextD} is a key ingredient in the proof of Theorem \ref{maim}. Observe that the extension depends on a parameter $\gamma$, whose choice will be specified during the proof.
\begin{proof}  [Proof of Theorem \ref{maim}] Our aim is to find $u := v - \widetilde{v_*} \in {\mathcal V}$ such that $u$ solves \eqref{wform}. In a first step, we fix $R > \delta(\mathcal S)$ and solve the truncated problem associated with \eqref{wform} in the bounded domain $\Omega_R$. Since $\mathrm{Re} \, \widetilde{v_*} \otimes \widetilde{v_*} - 2  {\mathbb D}(\widetilde{v_*} ) \in L^2(\Omega_R)^{3 \times 3}$, the map
$$
\begin{aligned}
{\mathcal V}_R \ni \varphi \longmapsto &  \, \mathrm{Re} \int_{\Omega_R} \widetilde{v_*} \cdot \nabla \varphi \cdot \widetilde{v_*} \dx - 2 \int_{\Omega_R} {\mathbb D}(\widetilde{v_*}): {\mathbb D}({\varphi})\dx \\
  & \, =  \mathrm{Re} \int_{\Omega_R} ( \widetilde{v_*} \otimes \widetilde{v_*} ) : {\mathbb D}({\varphi}) \dx - 2 \int_{\Omega_R} {\mathbb D}(\widetilde{v_*}): {\mathbb D}({\varphi})\dx  \\
  & \, =  \int_{\Omega_R} \left(  \mathrm{Re} \, \widetilde{v_*} \otimes \widetilde{v_*} - 2  {\mathbb D}(\widetilde{v_*} )\right) : {\mathbb D}({\varphi}) \dx   \end{aligned}
$$
defines a linear continuous functional on $\mathcal{V}_R$. Therefore, by Riesz representation Theorem, there exists $\mathcal{F}(\widetilde{v_*}) \in \mathcal{V}_R$ such that 
\begin{equation}
\left(\mathcal{F} (\widetilde{v_*}), \varphi \right)_{{\mathcal V}_R} = \int_{\Omega_R} \left(  \mathrm{Re} \, \widetilde{v_*} \otimes \widetilde{v_*} - 2  {\mathbb D}(\widetilde{v_*} )\right) : {\mathbb D}({\varphi}) \dx ,
\label{defF}
\end{equation}
\begin{equation}
\| \mathcal{F} (\widetilde{v_*}) \|_{\mathcal{V}_R} \leq C(\Omega) (1 + \mathrm{Re} )  \| \widetilde{v_*} \|_{W^{1,2}(\Omega)}.
\label{estF}
\end{equation}
For a fixed $u \in \mathcal{V}_R$, with $u_{\mathcal S}(x) = \xi + \omega \times x = V(x)$,  the map
$$
\begin{aligned}
{\mathcal V}_R \ni \varphi \longmapsto\   & \mathrm{Re} \int_{\Omega_R} \widetilde{v_*} \cdot \nabla \varphi \cdot u \dx + \mathrm{Re} \int_{\Omega_R}(u  - V) \cdot \nabla \varphi \cdot \widetilde{v_*} \dx \\
 & - \mathrm{Re} \int_{\Omega_R} ( \omega \times \widetilde{v_*}) \cdot \varphi \dx  
 \end{aligned}
$$
defines a linear continuous functional on $\mathcal{V}_R$. Again, by Riesz representation Theorem, there exists $\mathcal{K}(u,\widetilde{v_*})$ such that 
\begin{equation}
\begin{aligned}
\left(\mathcal{K}(u,\widetilde{v_*}), \varphi \right)_{{\mathcal V}_R} = & \, \mathrm{Re} \int_{\Omega_R} \widetilde{v_*} \cdot \nabla \varphi \cdot u \dx + \mathrm{Re} \int_{\Omega_R}(u - V) \cdot \nabla \varphi \cdot \widetilde{v_*} \dx \\ 
& - \mathrm{Re} \int_{\Omega_R} ( \omega \times \widetilde{v_*}) \cdot \varphi \dx  \, .
\end{aligned}
 \label{KR}
 \end{equation}
along with the estimate
\begin{equation}
\| \mathcal{K}(u,\widetilde{v_*}) \|_{{\mathcal V}_R}  \leq C(\Omega) \mathrm{Re} \| \widetilde{v_*} \|_{W^{1,2}(\Omega)} \| u \|_{{\mathcal V}_R}  \, .
\label{estK}
\end{equation}
Analogously, for a fixed $u \in \mathcal{V}_R$,  there exists $\mathcal{G}(u,u)$ such that
\begin{equation}
\begin{aligned}
\left( \mathcal{G}(u,u) , \varphi \right)_{{\mathcal V}_R}    =  & \, \mathrm{Re} \int_{\Omega_R}(u - V) \cdot \nabla \varphi \cdot u \dx -  \mathrm{Re}  \int_{\Omega} ( \omega \times u ) \cdot \varphi \dx \\
 & + \mathrm{Re} \, m \, \xi \times \omega \cdot a_\varphi  + \mathrm{Re} \, (J\omega)\times\omega \cdot b_\varphi \, ,
 \end{aligned}
 \label{GR}
\end{equation}
and
\begin{equation}
\| \mathcal{G}(u,u) \|_{{\mathcal V}_R}  \leq C(\Omega) \| u \|^2_{{\mathcal V}_R}  \, .
\label{estG}
\end{equation}
Hence, the identity \eqref{wform} restricted to $\Omega_R$ may be written as
\begin{equation}
\left(u_R + \mathcal{G}(u_R,u_R) + \mathcal{K}(u_R,\widetilde{v_*}) - \mathcal{F}(\widetilde{v_*}), \varphi \right)_{{\mathcal V}_R} = 0, \quad \forall\  \varphi \in {\mathcal V}_R \, ,
\label{wfvR}
\end{equation}
and, using the operator $A^{\frac 12} $ and  \eqref{a12}, we have
$$
\left(A^{\frac 12} u_R + A^{\frac 12}  \mathcal{G}(u_R,u_R) + A^{\frac 12}  \mathcal{K}(u_R,\widetilde{v_*}) - A^{\frac 12} \mathcal{F}(\widetilde{v_*}), A^{\frac 12}  \varphi \right)_{{\mathcal H}_R} = 0, \quad \forall\  \varphi \in {\mathcal V}_R \, .
$$

Now, defining
$$
U_R := A^{\frac 12} u_R,  \qquad \phi := A^{\frac 12}  \varphi ,
$$
we have 
$$
 \left(U_R + A^{\frac 12}  \mathcal{G}(A^{-\frac 12} U_R,A^{-\frac 12} U_R) + A^{\frac 12}  \mathcal{K}(A^{-\frac 12} U_R,\widetilde{v_*})  - A^{\frac 12} \mathcal{F}(\widetilde{v_*}),  \phi \right)_{{\mathcal H}_R} = 0, \quad \forall \phi \in {\mathcal H}_R \, .
$$

By means of the operator $$B : {\mathcal H}_R  \longrightarrow {\mathcal H}_R, \qquad B U  :=  A^{\frac 12} \mathcal{F}(\widetilde{v_*}) - A^{\frac 12}  \mathcal{K}(A^{-\frac 12} U,\widetilde{v_*})  - A^{\frac 12}  \mathcal{G}(A^{-\frac 12} U,A^{-\frac 12} U) $$ 
we are led to find a solution $U_R  \in {\mathcal H}_R$ of the following nonlinear operator equation
\begin{equation}\label{oseenquasi2}
	U_R = B U_R  \quad \text{in} \ \ {\mathcal H}_R \, .
\end{equation}
One can show that the operator $B$ satisfies the following properties. 
\begin{enumerate}
\item From \eqref{na12} and the estimates \eqref{estF}, \eqref{estK} and \eqref{estG}, we obtain 
$$
\begin{aligned}
\|B U \|_{{\mathcal H}_R} \leq  & \,  \| A^{\frac 12} \mathcal{F}(\widetilde{v_*}) \|_{{\mathcal H}_R} + \| A^{\frac 12}  \mathcal{K}(A^{-\frac 12} U,\widetilde{v_*}) \|_{{\mathcal H}_R}
+ \| A^{\frac 12}  \mathcal{G}(A^{-\frac 12} U,A^{-\frac 12} U) \|_{{\mathcal H}_R} \\ & =  \| 
 \mathcal{F}(\widetilde{v_*}) \|_{{\mathcal V}_R} + \| \mathcal{K}(A^{-\frac 12} U,\widetilde{v_*}) \|_{{\mathcal V}_R} + \| \mathcal{G}(A^{-\frac 12} U,A^{-\frac 12} U) \|_{{\mathcal V}_R}  \\
 & \leq \,   \| 
 \mathcal{F}(\widetilde{v_*}) \|_{{\mathcal V}_R} + C(\Omega) \mathrm{Re} \left( \| \widetilde{v_*} \|_{W^{1,2}(\Omega)} + \| A^{-\frac 12} U \|_{{\mathcal V}_R} \right) \| A^{-\frac 12} U \|_{{\mathcal V}_R}   \\
 & = \,   \| 
 \mathcal{F}(\widetilde{v_*}) \|_{{\mathcal V}_R} + C(\Omega) \mathrm{Re} \left( \| \widetilde{v_*} \|_{W^{1,2}(\Omega)} + \| U \|_{{\mathcal H}_R} \right) \| U \|_{{\mathcal H}_R}  ,
 \end{aligned}
 $$
 for all $U \in {\mathcal H}_R$.
 Moreover,  for $U,W \in {\mathcal H}_R$, we have
 $$
 \begin{aligned}
 \|B U - B W \|_{{\mathcal H}_R}  \leq  & \, \| A^{\frac 12}  \mathcal{K}(A^{-\frac 12} (U-W),\widetilde{v_*})  \|_{{\mathcal H}_R}  \\
 & + \|  A^{\frac 12}  \mathcal{G}(A^{-\frac 12} U,A^{-\frac 12} U) - A^{\frac 12}  \mathcal{G}(A^{-\frac 12} W,A^{-\frac 12} W) \|_{{\mathcal H}_R} \\
 := & \, S_1 + S_2.
  \end{aligned}
 $$
We set $u:=A^{-\frac 12}U,w:=A^{-\frac 12} W \in {\mathcal V}_R$. Using the relation \eqref{A-12} and \eqref{estK}, we can estimate $S_1$ in the following way
 $$
 \begin{aligned}
 S_1 \leq & \, C(\Omega) \mathrm{Re}  \| \widetilde{v_*} \|_{W^{1,2}(\Omega)}  \| u  - w  \|_{{\mathcal V}_R}  \\
 =  & \, C(\Omega) \mathrm{Re}  \| \widetilde{v_*} \|_{W^{1,2}(\Omega)}  \| A^{-1/2}U - A^{-1/2} W  \|_{{\mathcal V}_R}  \\
 = &  \, C(\Omega) \mathrm{Re}  \| \widetilde{v_*} \|_{W^{1,2}(\Omega)}  \| U - W  \|_{{\mathcal H}_R}.
  \end{aligned}
 $$
Similarly, we can estimate 
 $$
 \begin{aligned}
 S_2 = & \,  \|  \mathcal{G}(A^{-\frac 12} U,A^{-\frac 12} U) - \mathcal{G}(A^{-\frac 12} W,A^{-\frac 12} W) \|_{{\mathcal V}_R} \\
 = & \,  \|  \mathcal{G}(u,u) - \mathcal{G}(w,w) \|_{{\mathcal V}_R} \\
 \leq & \, C(\Omega_R) \left( \| u \|_{{\mathcal V}_R}  +  \| w \|_{{\mathcal V}_R}  \right) \| u - w \|_{{\mathcal V}_R}  \\ 
  = & \, C(\Omega_R) \left( \| U \|_{{\mathcal H}_R}  +  \| W \|_{{\mathcal H}_R}  \right) \| U - W \|_{{\mathcal H}_R} .
   \end{aligned}
   $$
   We conclude that
   \begin{equation}
    \|B U - B W \|_{{\mathcal H}_R}  \leq \left[ C(\Omega) \mathrm{Re}  \| \widetilde{v_*} \|_{W^{1,2}(\Omega)}  + C(\Omega_R) \left( \| U \|_{{\mathcal H}_R}  +  \| W \|_{{\mathcal H}_R}  \right) \right]  \| U - W \|_{{\mathcal H}_R} .
    \label{bubw}
   \end{equation}
 \item Let $\{U_j\}_{j \in \mathbb N} \in {\mathcal H}_R$ be a sequence which converges strongly to $U \in {\mathcal H}_R$. Then, from \eqref{bubw} we deduce
 \begin{equation}
 \|B U - B U_j\|_{{\mathcal H}_R}  \leq \left[ C(\Omega) \mathrm{Re}  \| \widetilde{v_*} \|_{W^{1,2}(\Omega)}  + C(\Omega_R) \left( \| U \|_{{\mathcal H}_R}  +  \| U_j \|_{{\mathcal H}_R}  \right) \right] \| U - U_j \|_{{\mathcal H}_R}.
  \label{UUj}
  \end{equation}
  Observe that $\|B U - B U_j\|_{{\mathcal H}_R}\to 0\text{ as } j \to \infty$.
 This shows that $B$ is a continuous operator.
\item To show that $B$ is completely continuous (see definition in \cite[Page 92]{Sohr2012}), we consider a bounded sequence $\{U_j\}_{j \in \mathbb N}$  in ${\mathcal H}_R$. Since $A^{-1/2}$ is compact,  the sequence $\{A^{-1/2}U_j\}_{j \in \mathbb N} \subset {\mathcal V}_R$ contains a strongly convergent subsequence $\{A^{-1/2}U_{j'}\}_{j' \in \mathbb N}$ in ${\mathcal V}_R$ and therefore this subsequence is a Cauchy sequence in  ${\mathcal V}_R$. On the other hand,
since $$\| U_{j'} - U_{k'} \|_{{\mathcal H}_R}  = \| A^{-1/2}U_{j'} -  A^{-1/2}U_{k'} \|_{{\mathcal V}_R} ,$$
it follows that $\{U_{j'}\}_{j' \in \mathbb N}$ is a Cauchy sequence in ${\mathcal H}_R$ and therefore it strongly converges to some $U$ in ${\mathcal H}_R$. From \eqref{UUj} it follows that
$$
\|B U - B U_{j'}\|_{{\mathcal H}_R}   \to 0, \text{ as } j \to \infty.
$$
and we conclude that $\{B U_j\}_{j \in \mathbb N}$ contains a strongly convergent subsequence. 
\end{enumerate}

By the Leray-Schauder Theorem (see \cite[3.1.1 Lemma, Page 93]{Sohr2012}), in order to prove that the problem in $\Omega_R$ possesses at least one solution, it suffices to guarantee that any $U_R^{\lambda} \in {\mathcal H}_R$ such that
\begin{equation}\label{oseenquasi3}
	U^{\lambda}_{R} = \lambda B U^{\lambda}_R  \quad \text{in} \ \ {\mathcal H}_R \, ,
\end{equation}
is uniformly bounded with respect to $\lambda \in [0, 1]$. 

Given $\lambda \in [0, 1]$ and $U_R \in \mathcal{H}_R \setminus \{0\}$ such that \eqref{oseenquasi3} holds, we have
$$
\left(U^{\lambda}_{R},U^{\lambda}_{R}\right)_{{\mathcal H}_R} = \lambda \left(B U^{\lambda}_R, U^{\lambda}_{R}\right)_{{\mathcal H}_R}.
$$
Setting $u^{\lambda}_R =  A^{- \frac 12} U_R \in \mathcal{V}_R \setminus \{0\}$, we can write
\begin{equation}
\| U^{\lambda}_{R} \|^2_{{\mathcal H}_R} = \left(U^{\lambda}_{R},U^{\lambda}_{R}\right)_{{\mathcal H}_R}  =   \left( A^{\frac 12} u^{\lambda}_R, A^{\frac 12} u^{\lambda}_R\right)_{{\mathcal H}_R}  =   \left(u^{\lambda}_R, u^{\lambda}_R\right)_{{\mathcal V}_R} =  2 \| {\mathbb D}(u^{\lambda}_{R})\|_{2,\Omega_{R}}^2,
\label{UAul}
\end{equation}
and, by \eqref{defF},  \eqref{KR}  and  \eqref{GR}, we get
$$
\begin{aligned}
 \lambda \left(B U^{\lambda}_R, U^{\lambda}_{R}\right)_{{\mathcal H}_R}
& =  \lambda \left(A^{\frac 12} \mathcal{F}(\widetilde{v_*}) - A^{\frac 12}  \mathcal{K}(A^{-\frac 12} U^{\lambda}_R,\widetilde{v_*})
- A^{\frac 12}  \mathcal{G}(A^{-\frac 12} U^{\lambda}_R,A^{-\frac 12} U^{\lambda}_R), U^{\lambda}_{R}\right)_{{\mathcal H}_R} \\
& =  \lambda \left(A^{\frac 12} \mathcal{F}(\widetilde{v_*}) - A^{\frac 12}  \mathcal{K}(u^{\lambda}_R,\widetilde{v_*})  - A^{\frac 12}  \mathcal{G}(u^{\lambda}_R, u^{\lambda}_R), A^{\frac 12}u^{\lambda}_{R}\right)_{{\mathcal H}_R} \\
& =  \lambda \left( \mathcal{F}(\widetilde{v_*})  - \mathcal{K}(u^{\lambda}_R,\widetilde{v_*})  -  \mathcal{G}(u^{\lambda}_R, u^{\lambda}_R), u^{\lambda}_{R}\right)_{{\mathcal V}_R} \\
& = \lambda  \mathrm{Re} \int_{\Omega_R} (u^{\lambda}_R - V^{\lambda}_R) \cdot \nabla u^{\lambda}_R \cdot u^{\lambda}_R \dx -  \lambda  \mathrm{Re}  \int_{\Omega_R} ( \omega^\lambda_R \times u^{\lambda}_R ) \cdot u^{\lambda}_R\dx \\
 & + \lambda  \mathrm{Re} \, m \, \xi_R^\lambda \times \omega_R^\lambda \cdot \xi^{\lambda}_R + \lambda  \mathrm{Re} \, (J\omega_R^\lambda)\times\omega_R^\lambda \cdot \omega^{\lambda}_R + \lambda \mathrm{Re} \int_{\Omega_R} \widetilde{v_*} \cdot \nabla u^{\lambda}_R\cdot u^{\lambda}_R \dx \\
 & + \lambda  \mathrm{Re} \int_{\Omega}(u^{\lambda}_R - V^{\lambda}_R) \cdot \nabla u^{\lambda}_R \cdot \widetilde{v_*} \dx  - \lambda  \mathrm{Re} \int_{\Omega_R} ( \omega_R^\lambda \times \widetilde{v_*}) \cdot u^{\lambda}_R \dx \\
 &
  + \lambda  \mathrm{Re} \int_{\Omega_R} \widetilde{v_*} \cdot \nabla u^{\lambda}_R \cdot \widetilde{v_*} \dx - 2 \lambda  \int_{\Omega_R} {\mathbb D}(\widetilde{v_*}): {\mathbb D}(u^{\lambda}_R)\dx .
\end{aligned}
$$
Due to the definition of ${\mathcal V}_R$, using integration by parts, we obtain 
\begin{equation}\label{reduce1}
\int_{\Omega_R} (u^{\lambda}_R - V^{\lambda}_R) \cdot \nabla u^{\lambda}_R \cdot u^{\lambda}_R \dx=0, \forall u^{\lambda}_R \in {\mathcal V}_R.    
\end{equation}
Moreover, 
\begin{equation}\label{reduce2}
( \omega^\lambda_R \times u^{\lambda}_R ) \cdot u^{\lambda}_R = 0,\quad  m \, \xi_R^\lambda \times \omega_R^\lambda \cdot \xi^{\lambda}_R=0,\quad (J\omega_R^\lambda)\times\omega_R^\lambda \cdot \omega^{\lambda}_R=0.
\end{equation}
Using \eqref{reduce1}-\eqref{reduce2}, we obtain
\[ 
\begin{aligned}
 \lambda \left(B U^{\lambda}_R, U^{\lambda}_{R}\right)_{{\mathcal H}_R} = & \,  \lambda \mathrm{Re} \int_{\Omega_R} \widetilde{v_*} \cdot \nabla u^{\lambda}_R\cdot u^{\lambda}_R \dx \\
 & + \lambda  \mathrm{Re} \int_{\Omega}(u^{\lambda}_R - V^{\lambda}_R) \cdot \nabla u^{\lambda}_R \cdot \widetilde{v_*} \dx  - \lambda  \mathrm{Re} \int_{\Omega_R} ( \omega_R^\lambda \times \widetilde{v_*}) \cdot u^{\lambda}_R \dx \\
 &
  + \lambda  \mathrm{Re} \int_{\Omega_R} \widetilde{v_*} \cdot \nabla u^{\lambda}_R \cdot \widetilde{v_*} \dx - 2 \lambda  \int_{\Omega_R} {\mathbb D}(\widetilde{v_*}): {\mathbb D}(u^{\lambda}_R)\dx  \\
=  & \, 2 \,  \lambda \mathrm{Re}  \int_{\Omega_{R}} (u^\lambda_{R} - V^\lambda_{R}) \cdot {\mathbb W}(u^\lambda_{R}) \cdot \widetilde{v_*} \dx   - 2 \lambda \int_{\Omega_{R}} {\mathbb D}(\widetilde{v_*}): {\mathbb D}(u^\lambda_{R})\dx  \\
& + \lambda \mathrm{Re}  \int_{\Omega_R} \widetilde{v_*} \cdot \nabla u^\lambda_{R} \cdot \widetilde{v_*} \dx.   
\end{aligned}
 \]
We conclude that
\[ 
\begin{aligned}
2 \| {\mathbb D}(u^\lambda_{R})\|_{2,\Omega_{R}}^2   
 = & \, 
  2 \,  \lambda \mathrm{Re}  \int_{\Omega_{R}} (u^\lambda_{R} - V^\lambda_{R}) \cdot {\mathbb W}(u^\lambda_{R}) \cdot \widetilde{v_*} \dx  \\  & - 2 \lambda  \int_{\Omega_{R}} {\mathbb D}(\widetilde{v_*}): {\mathbb D}(u^\lambda_{R})\dx  +  \mathrm{Re} \lambda  \int_{\Omega_R} \widetilde{v_*} \cdot \nabla u^\lambda_{R} \cdot \widetilde{v_*} \dx.   
\end{aligned}
 \]
Combining this identity with Lemma \ref{lemaextD}, we obtain the  estimate
\[ 
\begin{aligned}
    \| u^\lambda_{R}  \|_{{\mathcal V}_{R}}^2  \leq & \left(\gamma +  C_0(\Omega) \,  \mathrm{Re} \, | \Phi  | \right) \| u^\lambda_{R} \|_{{\mathcal V}_{R}}^2 \\
    &  +   C_1(\Omega) \,  \mathrm{Re}  \|\widetilde{v_*}\|_{1,2,\Omega} \| u^\lambda_{R} \|_{{\mathcal V}_R} +  C_2(\Omega) \,  \mathrm{Re} \|\widetilde{v_*}\|^2_{1,2,\Omega} \| u^\lambda_{R} \|_{{\mathcal V}_{R}},
\end{aligned}
 \]
for all $R > \delta(\mathcal S)$ and $\lambda \in [0,1]$. Note that the extension $\widetilde{v_*}$ depends on $\gamma$. Suppose  $ C_0(\Omega) \,  \mathrm{Re} \, | \Phi  |  <1$. Then it is possible to choose $\gamma >0$ and the corresponding extension $\widetilde{v_*}$ so that
$$
1 - \left(\gamma +  C_0(\Omega) \,  \mathrm{Re} \, | \Phi  | \right) > 0
$$
and this yields the uniform bound
\begin{equation}
\| u^\lambda_{R} \|_{{\mathcal V}_{R}}  \leq  \frac{  \mathrm{Re} }{1 - \left(\gamma +  C_0(\Omega) \,  \mathrm{Re} \, | \Phi  | \right)} \left( C_1(\Omega)\|\widetilde{v_*}\|_{1,2,\Omega}  + C_2(\Omega) \|\widetilde{v_*}\|^2_{1,2,\Omega} \right),
\label{uRk}
\end{equation}
for all $R > \delta(\mathcal S)$ and $\lambda \in [0,1]$. Recalling that $
U_R^\lambda = A^{1/2}u^\lambda_{R}$ and relations \eqref{UAul}, from \eqref{uRk} we get
\begin{equation}
\| U^\lambda_{R} \|_{{\mathcal H}_{R}}  \leq  \frac{  \mathrm{Re} }{1 - \left(\gamma +  C_0(\Omega) \,  \mathrm{Re} \, | \Phi  | \right)} \left( C_1(\Omega)\|\widetilde{v_*}\|_{1,2,\Omega}  + C_2(\Omega) \|\widetilde{v_*}\|^2_{1,2,\Omega} \right),
\label{Uk}
\end{equation}
for all $R > \delta(\mathcal S)$ and $\lambda \in [0,1]$.

Hence, the Leray-Schauder Theorem yields existence of a solution $U_R \in {\mathcal H}_R$ to equation \eqref{oseenquasi2}. This, in turn, is equivalent to the existence of a solution $u_R \in {\mathcal V}_R$ to \eqref{wfvR}, which is the truncation of the weak problem \eqref{wform} to the domain $\Omega_R$. This approach is then combined with the invading domains technique, as in \cite{Hey,GS2002}, to generate a new sequence of approximating solutions. We take a sequence of bounded subdomains of $\Omega$ such that \[\Omega = \cup_{j=1}^\infty \Omega_{R_j},\] and  construct a sequence  $\{(u_{R_j}, \xi_{R_j}, \omega_{R_j})\}_{j \in \mathbb N}$, such that
\begin{equation}
\| u_{R_j} \|_{{\mathcal V}_{R_j}}  \leq  \frac{  \mathrm{Re} }{1 - \left(\gamma +  C_0(\Omega) \,  \mathrm{Re} \, | \Phi  | \right)} \left( C_1(\Omega)\|\widetilde{v_*}\|_{1,2,\Omega}  + C_2(\Omega) \|\widetilde{v_*}\|^2_{1,2,\Omega} \right), \quad \forall j \in {\mathbb N}.
\label{uuk}
\end{equation}
Since the estimate \eqref{uuk} is uniform in $j$, passing to the limit $j \to \infty$ along a suitable subsequence, yields the desired weak solution $(u,V)$ for the exterior domain $\Omega$, which satisfies \eqref{uuk} with $u_{R_j}$ replaced by $u$.

Suppose that $
 \mathrm{Re} \, C_0(\Omega) | \Phi  | \leq 1/2
$ and take $0 < \gamma < 1/2$. From Lemma \ref{lemaextD}, if $\| v_{*} \|_{1/2,2,\partial \Omega} \leq M$, for some $M>0$, there is $C(\Omega,\gamma,M)>0$ such that
\[
\| \widetilde{v_{*}} \|_{1,2,\Omega} \leq C(\Omega,\gamma,M)\| v_{*} \|_{1/2,2,\partial \Omega}
\]
and consequently, the following estimate holds for $u$
\[
\| u\|_{\mathcal V}  \leq   C(M,\Omega)  \, \mathrm{Re} \, ( \|v_*\|_{1/2,2,\Omega} + \|v_*\|^2_{1/2,2,\Omega} ) .
\]
The estimates for $\|\nabla u\|_{2,\Omega}$ and for $|\xi|$ and $|\omega|$ follow from \eqref{kor} and \eqref{abd}. Finally, combine the estimates for $\|\nabla u\|_{2,\Omega}$ and $\|\nabla \widetilde{v_{*}}\|_{2,\Omega}$ to obtain the desired estimate for $\|\nabla v\|_{2,\Omega}$.
\end{proof}

\end{document}